\documentclass[10pt]{amsart}
\usepackage{amssymb}
\usepackage{bm}
\usepackage{graphicx}
\usepackage[centertags]{amsmath}
\usepackage{amsfonts}
\usepackage{amsthm}
\usepackage{amsbsy}
\usepackage{mathtools}
\usepackage{mathrsfs}
\usepackage{cases}
\usepackage{xfrac}
\usepackage[all]{xy}
\usepackage[linktocpage=true]{hyperref}
\usepackage{hyperref}
\hypersetup{
    colorlinks=true,
    linkcolor=blue,
    filecolor=blue,
    citecolor=blue,
    urlcolor=cyan}
\usepackage[margin=4cm]{geometry}
\newtheorem{thm}{Theorem}[section]
\newtheorem{cor}[thm]{Corollary}

\newtheorem{prop}[thm]{Proposition}

\newtheorem{fact}[thm]{Fact}

\theoremstyle{remark}
\newtheorem{rem}[thm]{Remark}

\newcommand{\rr}{\mathbb{R}}
\newcommand{\nn}{\mathbb{N}}
\newcommand{\zz}{\mathbb{Z}}
\newcommand{\ee}{\varepsilon}

\newcommand{\meg}{\geqslant}
\newcommand{\mik}{\leqslant}
\newcommand{\ave}{\mathbb{E}}

\begin{document}

\title{Uniformity norms, their weaker versions, and applications}

\author{Pandelis Dodos and Vassilis Kanellopoulos}

\address{Department of Mathematics, University of Athens, Panepistimiopolis 157 84, Athens, Greece}
\email{pdodos@math.uoa.gr}

\address{National Technical University of Athens, Faculty of Applied Sciences,
Department of Mathematics, Zografou Campus, 157 80, Athens, Greece}
\email{bkanel@math.ntua.gr}



\begin{abstract}
We show that, under some mild hypotheses, the Gowers uniformity norms (both in the additive and in the hypergraph setting)
are essentially equivalent to certain weaker norms which are easier to understand. We present two applications of this equivalence:
a variant of the Koopman--von Neumann decomposition, and a proof of the relative inverse theorem for the Gowers $U^s[N]$-norm using
a norm-type pseudorandomness condition.
\end{abstract}


\maketitle

\section{Introduction}

\numberwithin{equation}{section} \label{sec1}

\subsection{Overview} \label{subsec1.1}

This note is motivated by problems in arithmetic combinatorics and related parts of Ramsey theory, and focuses on the
relation between two notions of pseudorandomness which appear in this context. The first notion is measured using the
\textit{Gowers uniformity norms} \cite{Go1,Go2}. These norms are very useful in order to  accurately count the number 
of copies of certain ``patterns" in subsets of discrete structures; see, \textit{e.g.}, \cite[Lemma 11.4]{TV}.
However, they are defined by estimating the correlation of a function with shifts of itself, and so their dual norms 
are hopelessly difficult to understand in full~generality.

To compensate this problem, one adopts a functional analytic point of view. 
First one selects a class $\mathcal{D}$ of bounded functions (the ``dual" functions), 
and then associates with $\mathcal{D}$ a norm defined by the rule
$\|f\|_{\mathcal{D}}\coloneqq \sup\big\{ |\langle f,g\rangle|: g\in\mathcal{D}\big\}$. 
If the set $\mathcal{D}$ is appropriately selected, then the norm $\|\cdot\|_{\mathcal{D}}$ 
is comparable to the Gowers uniformity norm for \textit{bounded} functions. Unfortunately, in general,
the norm $\|\cdot\|_{\mathcal{D}}$ is significantly weaker, and this apparently excludes its applicability 
in the study of sparse sets like the set of primes~numbers.

Nevertheless, relatively recently it was shown, first implicitly in \cite{CFZ} and then more explicitly 
in \cite{TZ1,TZ2}, that the Gowers uniformity norms and their aforementioned weaker versions are essentially 
equivalent for a fairly large (and practically useful) family of unbounded functions.

We analyze further this phenomenon (both in the additive and in the hypergraph setting) and we show that
it is more typical than anticipated. Compared with the results in \cite{CFZ,TZ1,TZ2} which rely on the 
``linear forms condition" (a pseudorandomness hypothesis originating from \cite{GT1}), our approach is 
more intrinsic and is based exclusively on the properties of the Gowers uniformity norms. 
In a nutshell, our main results (Propositions \ref{p2.1} and \ref{p3.1}) follow from the Gowers--Cauchy--Schwarz
inequality and a simple decomposition method introduced in \cite{DKK1}.

\subsection{Applications} \label{subsec1.2}

We present two applications of this equivalence. 

The first application---Corollary \ref{c4.4} in Section \ref{sec4}---is a variant of the Koopman--von Neumann decomposition.
It answers a question of Gowers (see page 37 in the arXiv version of \cite{Go3}), and it asserts that a real-valued 
function $f$ on a finite additive group can be approximated in a Gowers uniformity norm by a bounded function, 
provided that $|f|$ is majorized by a function $\nu$ which satisfies a natural norm-type pseudorandomness condition. 
It is important to note that, besides its intrinsic interest, this approximation
is essential for further applications. Indeed, Corollary \ref{c4.4} together with an appropriate 
version of the generalized von Neumann theorem---\textit{e.g.}, \cite[Proposition~5.3]{GT1}---provides yet 
another approach\footnote{See \cite{CFZ,Go3,Tao,Zh} for other proofs of the relative Szemer\'{e}di theorem.} 
to the relative Szemer\'{e}di theorem \cite[Theorem 3.5]{GT1}, one of the main two ingredients of 
the Green--Tao theorem \cite{GT1}. Corollary \ref{c4.4} was also recently used by Bienvenu, Shao 
and Ter\"{a}v\"{a}inen \cite{BST} in order to prove Green--Tao type results for certain sparse 
subsets of the primes which are of arithmetic interest.

The second application---which is presented in Section \ref{sec5}---is a proof of the relative inverse 
theorem for the Gowers $U^s[N]$-norm, a result which is part of the \textit{nilpotent Hardy--Littlewood method} 
invented by Green and Tao \cite{GT2}. Our approach is based on Corollary \ref{c4.4} and, as such, 
it shows that the relative inverse theorem can also be applied under a norm-type pseudorandomness condition.
(See also \cite[Theorem 8.1]{TT} for a recent quantitative refinement of this approach.)

\subsection{Notation} \label{subsec1.3}

For every positive integer $n$ we set $[n]\coloneqq \{1,\dots,n\}$, and for every nonempty finite set $V$ 
by $|V|$ we denote its cardinality. Moreover, for every function $f\colon V\to\rr$ by $\ave[f(v)\, |\, v\in V]$ 
we denote the average of~$f$, that~is,
\[ \ave[f(v)\, |\, v\in V] \coloneqq \frac{1}{|V|} \sum_{v\in V} f(v). \]
We also write $\ave_{v\in V} f(v)$ to denote the average of $f$, or simply $\ave[f]$
if the set $V$ is understood from the context.

We use the following $o(\cdot)$ and $O(\cdot)$ notation. If $a_1,\dots,a_k$ are parameters and $\eta$ is a positive real,
then we write $o_{\eta\to 0;a_1,\dots,a_k}(X)$ to denote a quantity bounded in magnitude by $X F_{a_1,\dots,a_k}(\eta)$
where $F_{a_1,\dots,a_k}$ is a function which depends on $a_1,\dots,a_k$ and goes to zero as $\eta\to 0$. Similarly, by
$O_{a_1,\dots,a_k}(X)$ we denote a quantity bounded in magnitude by $X C_{a_1,\dots,a_k}$ where $C_{a_1,\dots,a_k}$
is a positive constant depending on the parameters $a_1,\dots,a_k$; we also write $Y\ll_{a_1,\dots,a_k}\!X$ or
$X\gg_{a_1,\dots,a_k}\!Y$ for the estimate $|Y|=O_{a_1,\dots,a_k}(X)$.


\section{The Gowers uniformity norm versus its weak version}

\numberwithin{equation}{section} \label{sec2}

\subsection{ \! } \label{subsec2.1}

Let $Z$ be a finite additive group and let $s\meg 2$ be an integer. Also let $f\colon Z\to\rr$ and recall
that the \textit{Gowers uniformity norm $\|f\|_{U^s(Z)}$ of $f$} is defined by the rule
\begin{equation} \label{e2.1}
\|f\|_{U^s(Z)} \coloneqq \ave\Big[ \prod_{\omega\in\{0,1\}^s}\!\! f(x+\omega\cdot \mathbf{h})\, \Big| \,
x\in Z, \mathbf{h}\in Z^s\Big]^{1/2^s}
\end{equation}
where $\omega\cdot \mathbf{h}\coloneqq \sum_{i=1}^s \omega_i\, h_i$ for every $\omega=(\omega_i)\in\{0,1\}^s$ 
and every $\mathbf{h}=(h_i)\in Z^s$. One can also define these norms for complex-valued functions by appropriately 
inserting complex conjugation operations---see \cite{TV} for details.

As we have noted, there is a natural weak version of the $U^s(Z)$-norm. Specifically, let $f\colon Z\to\rr$ and
define\footnote{There is no standard terminology for these norms.} the \textit{weak uniformity norm 
$\|f\|_{w^s(Z)}$ of $f$} by setting
\begin{equation} \label{e2.2}
\|f\|_{w^s(Z)} \coloneqq \sup\Big\{ \ave\big[ f(x)\!\!\!\!\!\!\!\prod_{\omega\in\{0,1\}^s\setminus \{0^s\}}\!\!\!\!\!\!
h_\omega(x+\omega\cdot \mathbf{h})\, \big| \, x\in Z, \mathbf{h}\in Z^s\big]\Big\}
\end{equation}
where the above supremum is taken over all families $\langle h_\omega:\omega\in \{0,1\}^s\setminus \{0^s\}\rangle$ 
of $[-1,1]$-valued functions on $Z$, and $0^s=(0,\dots,0)\in \{0,1\}^s$ denotes the sequence of length $s$ 
taking the constant value $0$. We observe that
\begin{equation} \label{e2.3}
\|f\|_{w^s(Z)}\mik \|f\|_{U^s(Z)}
\end{equation}
as can be seen by the Gowers--Cauchy--Schwarz inequality (see, \textit{e.g.}, \cite[(11.6)]{TV}).

\subsection{The main result} \label{subsec2.2}

By \eqref{e2.1} and \eqref{e2.2}, it follows readily that for every function $f\colon Z\to [-1,1]$ 
we have $\|f\|_{U^s(Z)}\mik \|f\|^{1/2^s}_{w^s(Z)}$. The following proposition shows that the estimate 
\eqref{e2.3} can also be \textit{reversed} provided that $f$ is merely bounded in magnitude
by a function $\nu\colon Z\to\rr^+$ satisfying a norm-type pseudorandomness condition.
\begin{prop} \label{p2.1}
Let $Z$ be a finite additive group, let $s\meg 2$ be an integer, and let $0<\eta\mik 1$. 
Also let $\nu\colon Z\to \rr^+$ such that
\begin{equation} \label{e2.4}
\|\nu-1\|_{U^{2s}(Z)}\mik \eta.
\end{equation}
Finally, let $f\colon Z\to\rr$ with $|f|\mik \nu$. If\, $\|f\|_{w^s(Z)}\mik \eta$, then
\begin{equation} \label{e2.5}
\|f\|_{U^s(Z)}= o_{\eta\to 0;s}(1).
\end{equation}
\end{prop}
Proposition \ref{p2.1} can be proved arguing as in \cite[Theorem 11]{TZ1} and using slightly stronger pseudorandomness
hypotheses (see also \cite[Proposition~3.7]{TZ2} for a variant of this argument). We will give a proof using as a main tool
the following simple consequence of the Gowers--Cauchy--Schwarz inequality for the $U^{2s}(Z)$-norm which was first
observed (in a slightly less general form) in the proof of Proposition 4 in~\cite{TZ1}.
\begin{fact} \label{f2.2}
Let $Z$ be a finite additive group, and let $s\meg 2$ be an integer. Also let $g\colon Z\to \rr$, let
$\langle g^{(k)}_\omega \colon k\in\{1,2\}, \omega \in \{0,1\}^s\setminus\{0^s\}\rangle$ be a family
of real-valued functions on $Z$, and set
\[ I\coloneqq \ave\Big[ g(x) \prod_{k\in\{1,2\}}\prod_{\omega\in \{0,1\}^s\setminus\{0^s\}} \!\!\!\!\!
g^{(k)}_\omega(x+\omega\cdot \mathbf{h}_k)\,\Big|\, x\in Z, \mathbf{h}_1, \mathbf{h}_2\in Z^s\Big].\]
Then we have
\[|I |\mik \|g\|_{U^{2s}(Z)}\, \cdot\! \prod_{k\in\{1,2\}}\prod_{\omega\in \{0,1\}^s\setminus\{ 0^s\}}
\!\! \|g^{(k)}_\omega\|_{U^{2s}(Z)}.\]
\end{fact}
\begin{proof}
We identify $\{0,1\}^{2s}$ with $\{0,1\}^s\times \{0,1\}^s$ and we write every $\omega\in \{0,1\}^{2s}$
as $\omega=(\omega_1,\omega_2)$ where $\omega_1,\omega_2\in\{0,1\}^s$. We define a family 
$\langle g_{\omega} \colon \omega\in \{0,1\}^{2s} \rangle$ of real-valued functions on $Z$
by setting: (i)  $g_{(0^s, 0^s)}=g$, (ii) $g_{(\omega, 0^s)}=g^{(1)}_\omega$ and  $g_{(0^s, \omega)}=g^{(2)}_\omega$
if $\omega\in \{0,1\}^s\setminus\{0^s\}$, and (iii) $g_{(\omega_1,\omega_2)}=1$ 
if $\omega_1, \omega_2\in \{0,1\}^s\setminus\{0^s\}$. Noticing that
\[ I=\ave\Big[ \prod_{\omega\in \{0,1\}^{2s}} 
g_\omega(x+\omega\cdot \mathbf{h})\,\Big|\, x\in Z, \, \mathbf{h}\in Z^{2s}\Big], \]
the result follows from the Gowers--Cauchy--Schwarz inequality.
\end{proof}
We proceed to the proof of Proposition \ref{p2.1}.
\begin{proof}[Proof of Proposition \emph{\ref{p2.1}}]
We will show that for every nonempty subset $\Omega$ of $\{0,1\}^s$  and for every (possibly empty) family
$\langle h_\omega:\omega\in \{0,1\}^s\setminus \Omega\rangle$ of $[-1,1]$-valued functions on $Z$ we
have\footnote{In \eqref{e2.6} we follow the convention that the product of an empty family of functions 
is equal to the constant function $1$.}
\begin{equation} \label{e2.6}
\ave\Big[ \prod_{\omega \in \Omega} f(x+\omega\cdot\mathbf{h})\!\!\! \prod_{\omega\in \{0,1\}^s\setminus \Omega}\!\!\!
h_\omega(x+\omega\cdot \mathbf{h}) \,\Big|\, x\in Z, \mathbf{h}\in Z^s  \Big]=o_{\eta\to 0;s}(1).
\end{equation}
Clearly, this is enough to complete the proof.

We proceed by induction on the cardinality of $\Omega$. Since the left-hand side of~\eqref{e2.6} is invariant
under permutations of the cube, the initial case $|\Omega|=1$ follows from our assumption that 
$\|f\|_{w^s(Z)}\mik \eta$. Next, let $m\in\{1,\dots, 2^s-1\}$ and assume that \eqref{e2.6} 
has been proved for every $\Omega\subseteq \{0,1\}^s$ with $|\Omega|=m$. 
Fix $\Omega'\subseteq \{0,1\}^s$ with $|\Omega'|=m+1$. By permuting the cube if necessary, 
we may assume that $0^s\in \Omega'$. Set $\Omega\coloneqq\Omega'\setminus \{0^s\}$ and notice that $|\Omega|=m$.
Also let $\langle h_\omega: \omega\in \{0,1\}^s\setminus \Omega'\rangle$ be an arbitrary family 
of $[-1,1]$-valued function on $Z$. We have to show that
\[ \ave \Big[ f(x) \, \prod_{\omega \in \Omega} f(x+\omega\cdot\mathbf{h}) \!\!\!\! 
\prod_{\omega\in \{0,1\}^s\setminus \Omega'}\!\!\!\!
h_\omega(x+\omega\cdot\mathbf{h}) \,\Big|\, x\in Z, \mathbf{h}\in Z^s\Big]=o_{\eta\to 0;s}(1) \]
or, equivalently,
\begin{equation} \label{e2.7}
\ave[f(x) G(x) \,|\, x\in Z]=o_{\eta\to 0;s}(1)
\end{equation}
where $G\colon Z\to\rr$ is the marginal defined by the rule
\begin{equation} \label{e2.8}
G(x)=\ave \Big[ \prod_{\omega \in \Omega} f(x+\omega\cdot \mathbf{h}) \!\!\!\! 
\prod_{\omega\in \{0,1\}^s\setminus \Omega'}\!\!\!\!
h_\omega(x+\omega\cdot\mathbf{h})\,\Big|\, \mathbf{h}\in Z^s\Big].
\end{equation}
Since $|f|\mik \nu$ and $\ave [\nu]\mik \|\nu\|_{U^{2s}(Z)}\mik 1+\eta$, 
by the Cauchy--Schwarz inequality, it is enough to prove that
\begin{equation} \label{e2.9}
\ave[(\nu-1)G^2]=o_{\eta\to 0;s}(1) \ \text{ and } \ \ave[G^2]=o_{\eta\to 0;s}(1).
\end{equation}
The first estimate in \eqref{e2.9} follows from Fact \ref{f2.2}
and the fact that $\|\nu-1\|_{U^{2s}(Z)}\mik \eta$; indeed, observe that
\[ |\ave[(\nu-1)G^2]|\mik \|\nu-1\|_{U^{2s}(Z)} \cdot \|f\|^{2|\Omega|}_{U^{2s}(Z)}\, \cdot \!\!\!\!\!
\prod_{\omega\in \{0,1\}^s\setminus \Omega'}\!\!\! \|h_\omega\|^2_{U^{2s}(Z)} \ll_s \eta. \]
For the second estimate, as in \cite[Theorem 7.1]{DKK1}, we will use a simple decomposition. 
Specifically, let $\beta>0$ be a cut-off parameter and write 
$G^2=\mathbf{1}_{[|G|\mik\beta]}G^2+\mathbf{1}_{[|G|>\beta]}G^2$. 
As we shall see, any value of $\beta$ greater than $1$ would suffice for the proof; 
for concreteness we will use the value $\beta=2$. By linearity of expectation, 
it is enough to show that
\begin{equation} \label{e2.10}
\ave[\mathbf{1}_{[|G|\mik 2]}G^2]=o_{\eta\to 0;s}(1) \ \text{ and } \ \ave[\mathbf{1}_{[|G|>2]}G^2]=o_{\eta\to 0;s}(1).
\end{equation}
The first part of \eqref{e2.10} can be handled easily by our inductive assumptions. Indeed, set
$h_{0^s}=\boldsymbol{1}_{[|G|\mik 2]}(G/2)$ and notice that
\[ \ave[\mathbf{1}_{[G|\mik 2]}G^2]=2 \cdot \ave\Big[h_{0^s}\big(x) 
\prod_{\omega \in \Omega} f(x+\omega\cdot\mathbf{h})\!\!\!\!\!
\prod_{\omega\in \{0,1\}^s\setminus \Omega'} \!\!\! 
h_\omega(x+\omega\cdot\mathbf{h}) \,\Big|\, x\in Z, \mathbf{h}\in Z^s\Big] \]
which is $o_{\eta\to 0;s}(1)$ since $|h_{0^s}|\mik 1$ and $\Omega\subseteq \{0,1\}^s$ satisfies $|\Omega|=m$.
For the second part of \eqref{e2.10}, observe that
\begin{equation} \label{e2.11}
\ave[\mathbf{1}_{[|G|> 2]} G^2]\mik \ave[\mathbf{1}_{[\mathcal{N}> 2]}\mathcal{N}^2]
\end{equation}
where $\mathcal{N}\colon Z\to\rr$ is defined by
$\mathcal{N}(x)=\ave[\, \prod_{\omega \in \Omega} \nu(x+\omega\cdot\mathbf{h})\,|\,\mathbf{h}\in Z^s]$.
The function $\mathcal{N}$ satisfies the following moment estimate: for every $A\subseteq Z$ and every 
$k\in \{1,2\}$ we have
\begin{equation} \label{e2.12}
 \ave[\mathbf{1}_A\mathcal{N}^k]=\mathbf{P}(A)+o_{\eta\to 0;s}(1)
\end{equation}
where $\mathbf{P}(A)=\ave[\mathbf{1}_A]=|A|/|Z|$ is the probability of $A$ with respect to 
the uniform probability measure $\mathbf{P}$ on $Z$. Indeed, since 
$|\ave[\mathbf{1}_A \mathcal{N}^k]-\mathbf{P}(A)|=|\ave[\mathbf{1}_A(\mathcal{N}^k-1)]|$, 
the estimate in \eqref{e2.12} follows from Fact \ref{f2.2}, a telescopic argument 
and the fact that $\|\nu-1\|_{U^{2s}(Z)}\mik \eta$. Now, combining \eqref{e2.11} and \eqref{e2.12}
for $k=2$ and invoking Markov's inequality, we have
\begin{eqnarray*}
\ave[\mathbf{1}_{[|G|> 2]}G^2] & \mik & \mathbf{P}\big([\mathcal{N}>2]\big) + o_{\eta\to 0;s}(1)
\mik  \mathbf{P}\big(\big[|\mathcal{N}-1|>1\big]\big) + o_{\eta\to 0;s}(1) \\
& \mik & \ave\big[|\mathcal{N}-1|\big]+ o_{\eta\to 0;s}(1).
\end{eqnarray*}
On the other hand, by \eqref{e2.12} for $k=1$, we see that
\[ \ave\big[|\mathcal{N}-1|\big]=\ave[\mathbf{1}_{[\mathcal{N}\meg 1]}(\mathcal{N}-1)]+
\ave[\mathbf{1}_{[\mathcal{N}< 1]}(1-\mathcal{N})]=o_{\eta\to 0;s}(1).\]
Therefore, $\ave[\mathbf{1}_{[|G|>2]}G^2]=o_{\eta\to 0;s}(1)$ as desired.
\end{proof}
\begin{rem} \label{r1}
It is not hard to see that the proof of Proposition \ref{p2.1} in fact yields that for every 
$0<\ee\mik 1$, if $\nu\colon Z\to\rr^+$ satisfies $\|\nu-1\|_{U^{2s}(Z)}\mik \eta$ for some 
$0<\eta\mik \ee$ and $f\colon Z\to\rr$ is such that $|f|\mik \nu$ and $\|f\|_{w^s(Z)}\mik \ee$,
then we have $\|f\|_{U^s(Z)}\ll_s \ee^{C}+o_{\eta\to 0}(1)$ where $C=(2^s\cdot 2^{2^s-1})^{-1}$.
\end{rem}
\begin{rem} \label{r2}
By appropriately modifying the proof of Proposition \ref{p2.1}, one can establish the equivalence 
between the $U^s(Z)$-norm and its weak version using more general pseudorandomness hypotheses. 
In particular, we have the following proposition which is related to \cite[Theorem 7.1]{DKK1}.
\begin{prop} \label{p2.3}
Let $Z, s, \eta$ be as in Proposition \emph{\ref{p2.1}}, let $1<p\mik \infty$, 
let $q$ denote the~con\-jugate exponent of $p$, and set  
$\ell \coloneqq \min \{2n: n\in\nn \text{ and } 2n\meg 2q\}$. Let $\nu\colon Z\to \rr^+$~such~that
\begin{equation} \label{e2.13}
\|\nu-\psi\|_{U^{\ell s}(Z)}\mik \eta
\end{equation}
where $\psi\colon Z\to\rr$ satisfies\footnote{Here, the $L_p$-norm of $\psi$ is computed using the uniform 
probability measure on $Z$, that is, $\|\psi\|_{L_p}\coloneqq \ave[\, |\psi(x)|^p\, |\, x\in Z]^{1/p}$.} 
$\|\psi\|_{L_p}\mik 1$ and $\|\psi\|_{U^{\ell s}(Z)}\mik 1$. Finally, let $f\colon Z\to \rr$ with $|f|\mik \nu$. 
If\, $\|f\|_{w^s(Z)}\mik \eta$, then
\begin{equation} \label{e2.14}
\|f\|_{U^s(Z)}= o_{\eta\to 0;s,p}(1).
\end{equation}
\end{prop}
Observe that Proposition \ref{p2.1} corresponds to the case ``$p=\infty$" and ``$\psi=1$"\!. Also note that, 
by H\"{o}lder's inequality,
if $p$ is sufficiently large, then the estimate $\|\psi\|_{U^{\ell s}(Z)}\mik 1$ follows from the estimate $\|\psi\|_{L_p}\mik 1$.
\end{rem}


\section{The box norm versus the cut norm} \label{sec3}

\numberwithin{equation}{section}

\subsection{ \! } \label{subsec3.1} 

Let $V$ be a nonempty finite set and let $s\meg 2$ be an integer. Also let $F\colon V^s\to\rr$ and recall that the
\textit{box norm $\|F\|_{\square(V^s)}$ of~$F$} is defined by the rule
\begin{equation} \label{e3.1}
\|F\|_{\square(V^s)}\coloneqq \ave\Big[ \prod_{\omega\in [2]^s} F\big( \pi_\omega(x)\big)\,\Big|\, x\in V^{s\times 2} \Big]^{1/2^s}
\end{equation}
where for every $\omega=(\omega_i)\in [2]^s$ by $\pi_\omega\colon V^{s\times 2}\to V^s$ we denote the projection
$\pi_\omega\big((x_{ij})\big)=(x_{i\,\omega_i})_{i=1}^s$. These norms are the abstract versions of the Gowers uniformity norms;
indeed, notice that for every finite additive group $Z$ and every $f\colon Z\to\rr$ we have
\begin{equation} \label{e3.2}
\|f\|_{U^s(Z)}=\|f(x_1+\cdots+x_s)\|_{\square(Z^s)}.
\end{equation}
We will also work with the following slight variants of the box norms which first appeared in \cite{Ha}: for every even integer
$\ell\meg 2$ we define the \textit{$\ell$-box norm $\|F\|_{\square_\ell(V^s)}$ of\, $F$} by setting
\begin{equation} \label{e3.3}
\|F\|_{\square_\ell(V^s)}\coloneqq \ave\Big[ \prod_{\omega\in [\ell]^s} 
F\big( \pi_\omega(x)\big)\,\Big|\, x\in V^{s\times \ell} \Big]^{1/\ell^s}
\end{equation}
where, as above, for every $\omega=(\omega_i)\in [\ell]^s$ by $\pi_\omega\colon V^{s\times \ell}\to V^s$ we denote the projection
$\pi_\omega\big((x_{ij})\big)=(x_{i\,\omega_i})_{i=1}^s$. Clearly, the $\square_2(V^s)$-norm coincides with the $\square(V^s)$-norm.
As the parameter $\ell$ increases, the quantity $\|F\|_{\square_\ell(V^s)}$ also increases and measures the integrability of $F$.
In particular, for bounded functions all these norms are essentially equivalent. This fact, together with some basic properties
of the $\ell$-box norms, are discussed in the appendix.

The box norm also has a natural weak version which is known as the \textit{cut norm} and originates from \cite{FK}.
Specifically, let $V, s$ and $F$ be as above, and define\footnote{In several places in the literature, the cut norm is defined
by taking the supremum in \eqref{e3.4} over all families $\langle H_\omega:\omega\in [2]^s\setminus \{1^s\}\rangle$ of
$[0,1]\text{-valued}$ functions on $V^s$. However, it is clear that this more restrictive definition yields an equivalent norm.}
the \textit{cut norm $\|F\|_{\mathrm{cut}(V^s)}$ of\, $F$} by the rule
\begin{equation} \label{e3.4}
\|F\|_{\mathrm{cut}(V^s)}\coloneqq \sup\Big\{ \ave\big[ F\big(\pi_{1^s}(x)\big)\!\!\!\!\! 
\prod_{\omega\in [2]^s\setminus \{1^s\}}\!\!\!
H_\omega\big( \pi_\omega(x)\big)\,\big|\, x\in V^{s\times 2} \big] \Big\}
\end{equation}
where the above supremum is taken over all families $\langle H_\omega:\omega\in [2]^s\setminus \{1^s\}\rangle$ of
$[-1,1]\text{-valued}$ functions on $V^s$, and $1^s=(1,\dots,1)\in [2]^s$ denotes the sequence of length $s$ 
taking the constant value $1$. By the Gowers--Cauchy--Schwarz inequality for the $\square(V^s)$-norm, 
\begin{equation} \label{e3.5}
\|F\|_{\mathrm{cut}(V^s)} \mik \|F\|_{\square(V^s)}.
\end{equation}
Also observe that if $F$ is $[-1,1]$-valued, then $\|F\|_{\square(V^s)}\mik \|F\|^{1/2^s}_{\mathrm{cut}(V^s)}$.

\subsection{The main result} \label{subsec3.2}

The following proposition is the analogue of Proposition~\ref{p2.1} and establishes the equivalence 
of the box norm with the cut norm.
\begin{prop} \label{p3.1}
Let $V$ be a nonempty finite set, let $s\meg 2$ be an integer, and let $0<\eta\mik 1$. 
Also let $\nu\colon V^s\to \rr^+$ such that
\begin{equation} \label{e3.6}
\|\nu-1\|_{\square_4(V^s)}\mik \eta.
\end{equation}
Finally, let $F\colon V^s\to\rr$ with $|F|\mik \nu$. If\, $\|F\|_{\mathrm{cut}(V^s)}\mik \eta$, then
\begin{equation} \label{e3.7}
\|F\|_{\square(V^s)}= o_{\eta\to 0;s}(1).
\end{equation}
\end{prop}
It is possible to prove Proposition \ref{p3.1} arguing as in \cite[Theorem 2.17]{CFZ}. 
However, as the reader has probably already noticed, Proposition \ref{p3.1} can be proved 
arguing precisely as in Proposition \ref{p2.1}, using instead of Fact \ref{f2.2} the
following elementary consequence of the  Gowers--Cauchy--Schwarz inequality for the $\square_4(V^s)$-norm.
\begin{fact} \label{f3.2}
Let $V$ be a nonempty finite set, and let $s\meg 2$ be an integer. Also let $G\colon V^s\to \rr$, let
$\langle G^{(k)}_\omega \colon k\in\{1,2\}, \omega \in [2]^s\setminus\{1^s\}\rangle$ be a family 
of real-valued functions on $V^s$, and set\footnote{Here, we identify $V^{s\times 2}$ with $V^s\times V^s$
via the bijection $V^{s\times 2}\! \ni (x_{ij})\mapsto \big((x_{i1}),(x_{i2})\big)\!\in V^s\times V^s$. 
In particular, we write uniquely every $x\in V^{s\times 2}$ as $x=(y,z)\in V^s\times V^s$.}
\[ I\coloneqq \ave\Big[ G(y) \prod_{k\in\{1,2\}}\prod_{\omega\in [2]^s\setminus\{1^s\}}\!\!
G^{(k)}_\omega\big(\pi_\omega(y,z_k)\big)\,\Big|\, y,z_1, z_2\in V^{s}\Big].\]
Then we have
\[|I |\mik \|G\|_{\square_4(V^s)}\, \cdot \! \prod_{k\in\{1,2\}}
\prod_{\omega\in [2]^s\setminus\{ 1^s\}} \!\! \|G^{(k)}_\omega\|_{\square_4(V^s)}.\]
\end{fact}
\begin{proof}
Define a map $\{1,3\}^s\setminus \{1^s\}\ni \omega=(\omega_i)\mapsto \omega'=(\omega'_i)\in [2]^s\setminus\{1^s\}$ 
by setting $\omega'_i=1$ if $\omega_i=1$, and $\omega'_i=2$ if $\omega_i=3$. Then we may write
\[ I=\ave\Big[ G_{1^s}\big(\pi_{1^s}(x)\big) \!\! \prod_{\omega\in [2]^s\setminus\{1^s\}} \!\!\!\!\! 
G_\omega\big(\pi_{\omega} (x)\big) \!\!
\prod_{\omega\in \{1,3\}^s\setminus\{1^s\}} \!\!\!\!\! 
G_\omega\big(\pi_{\omega} (x)\big)\,\Big|\,  x\in V^{s\times 4}\Big] \]
where we have $G_{1^s}= G$, $G_\omega= G^{(1)}_\omega$ for every 
$\omega\in [2]^s\setminus \{1^s\}$, and $G_\omega= G^{(2)}_{\omega'}$ for every
$\omega\in \{1,3\}^s\setminus\{1^s\}$. Thus, setting $G_{\omega}=1$ for all other $\omega\in [4]^s$, we see that
\[ I=\ave\Big[ \prod_{\omega\in [4]^s} G_\omega\big(\pi_\omega(x)\big)\,\Big|\, x\in V^{s\times 4}\Big] \]
and the result follows from the Gowers--Cauchy--Schwarz inequality.
\end{proof}
\begin{rem} \label{r3}
We point out that Proposition \ref{p2.3} can also be extended in the hypergraph setting. Specifically, 
we have the following proposition; see \cite[Section 7]{DKK1} for further results in this direction.
\begin{prop} \label{p3.3}
Let $V, s, \eta$ be as in Proposition \emph{\ref{p3.1}}, let $1<p\mik \infty$, let $q$ denote the conjugate 
exponent of $p$, and set $\ell\coloneqq \min\{2n: n\in\nn \text{ and } 2n\meg 2q+2\}$. 
Also let $\nu\colon V^s\to \rr^+$ such that
\begin{equation} \label{e3.8}
\|\nu-\psi\|_{\square_\ell(V^s)}\mik \eta
\end{equation}
where $\psi\colon V^s\to\rr$ satisfies\footnote{Here, as in Proposition \ref{p2.3}, the $L_p$-norm of $\psi$ 
is computed using the uniform probability measure on $V^s$, that is, 
$\|\psi\|_{L_p}\coloneqq \ave[\, |\psi(x)|^p\, |\, x\in V^s]^{1/p}$.} $\|\psi\|_{L_p}\mik 1$ and 
$\|\psi\|_{\square_\ell(V^s)}\mik 1$. Finally, let $F\colon V^s\to \rr$ with $|F|\mik \nu$.
If\, $\|F\|_{\mathrm{cut}(V^s)}\mik \eta$, then
\begin{equation} \label{e3.9}
\|F\|_{\square(V^s)}= o_{\eta\to 0;s,p}(1).
\end{equation}
\end{prop}
\end{rem}

\subsection{Transferring Proposition \ref{p3.1} to the additive setting} \label{subsec3.3}

There is an additive version of Proposition \ref{p3.1} which is somewhat distinct from 
Proposition \ref{p2.1} and is obtained by transferring the $\ell$-box norms and the cut norm 
in the additive setting via formula \eqref{e3.2}. Specifically, let $Z$ be a finite additive group, 
let $s\meg 2$ be an integer, and let $f\colon Z\to \rr$. For every even integer $\ell\meg 2$ we define
the \textit{$(s,\ell)$-uniformity norm $\|f\|_{U^s_\ell(Z)}$ of $f$} by
\begin{equation} \label{e3.10}
\|f\|_{U^s_\ell(Z)} \coloneqq \|f(x_1+\cdots+x_s)\|_{\square_\ell(Z^s)}.
\end{equation}
Respectively, we define the \textit{$s$-additive cut norm $\|f\|_{\mathrm{cut}^s(Z)}$ of $f$} by the rule
\begin{equation} \label{e3.11}
\|f\|_{\mathrm{cut}^s(Z)} \coloneqq \|f(x_1+\cdots+x_s)\|_{\mathrm{cut}(Z^s)}.
\end{equation}
(Notice that the additive cut norm is slightly stronger than the weak uniformity norm; in particular, we have
$\|f\|_{w^s(Z)}\mik \|f\|_{\mathrm{cut}^s(Z)}$.) Taking into account \eqref{e3.10} and \eqref{e3.11},
we see that Proposition \ref{p3.1} can be reformulated as follows.
\begin{cor} \label{c3.4}
Let $Z$ be a finite additive group, let $s\meg 2$ be an integer, and let\, $0<\eta\mik 1$. 
Also let $\nu\colon Z\to \rr^+$ such that
\begin{equation} \label{e3.12}
\|\nu-1\|_{U^s_4(Z)}\mik \eta.
\end{equation}
Finally, let $f\colon Z\to\rr$ with $|f|\mik \nu$. If\, $\|f\|_{\mathrm{cut}^s(Z)}\mik \eta$, then
\begin{equation} \label{e3.13}
\|f\|_{U^s(Z)}= o_{\eta\to 0;s}(1).
\end{equation}
\end{cor}


\section{A variant of the Koopman--von Neumann decomposition} \label{sec4}

\numberwithin{equation}{section}

\subsection{Overview} \label{subsec4.1}

The \emph{Koopman--von Neumann decomposition} is a circle of results asserting that, 
under certain circumstances, one can decompose a function $f$ as $f=f_{\mathrm{bnd}}+f_{\mathrm{err}}$ 
where $f_{\mathrm{bnd}}$ is bounded in magnitude by $1$ and $f_{\mathrm{err}}$ has small uniformity 
norm\footnote{As we have already noted in the introduction, in applications it is not enough to control
the error-term $f_{\mathrm{err}}$ using a weaker norm.}. To see the relevance in this context 
of the equivalence between the uniformity norms and their weaker versions, note that one can first 
approximate $f$ by a bounded function $f_{\mathrm{bnd}}$ such that the difference $f-f_{\mathrm{bnd}}$ 
is small in a weaker norm, and then upgrade this information using the results in the previous sections.
This strategy (also used in \cite{TZ1,TZ2}) is quite effective partly because the aforementioned 
weaker approximation can be achieved relatively easily using various methods. 
We will use one of these methods, the so-called \textit{dense model theorem}.

\subsection{Consequences of the dense model theorem} \label{subsec4.2}

We begin by recalling the dense model theorem; we will state the formulation which is closest 
to the purposes of this note (see \cite[Theorem 1.1]{RTTV} or \cite[Theorem 3.5]{TZ2}).
\begin{prop} \label{p4.1}
Let $X$ be a finite set, and let $\mathcal{F}$ be a family of $[-1,1]$-valued functions on $X$. 
Also let $0<\eta\mik 1$, and let $\nu\colon X\to \rr^+$ such that $\ave[\nu]\mik 1+\eta$ and satisfying
\begin{equation} \label{e4.1}
\Big|\ave \Big[(\nu-1)\prod_{i=1}^k F_i \Big]\Big|\mik \eta
\end{equation}
for every $F_1,\dots,F_k\in\mathcal{F}$. Then for every $g\colon X\to \rr$ with $0\mik g\mik \nu$
there exists $w\colon X\to [0,1]$ such that
\begin{equation} \label{e4.2}
\sup\big\{ |\ave\big[(g-w)F]|: F\in\mathcal{F}\big\}=o_{\eta\to 0}(1).
\end{equation}
\end{prop}
We will need two consequences of Proposition \ref{p4.1}. The first one concerns functions defined 
on a finite additive group $Z$. Recall that by $\|\cdot\|_{\mathrm{cut}^s(Z)}$ we denote the additive 
cut norm defined in \eqref{e3.11}.
\begin{cor} \label{c4.2}
Let $Z$ be a finite additive group, and let $s\meg 2$ be an integer. 
Let $0<\eta\mik 1$ and $\nu\colon Z\to \rr^+$ such that
$\|\nu-1\|_{\mathrm{cut}^s(Z)}\mik \eta$. Then for every $g\colon Z\to\rr$ with $0\mik g\mik \nu$ 
there exists $w\colon Z\to [0,1]$ such that $\|g-w\|_{\mathrm{cut}^s(Z)}= o_{\eta\to 0}(1)$. 
Consequently, for every $f\colon Z\to\rr$ with $|f|\mik \nu$ there exists $h\colon Z\to [-1,1]$ such that
$\|f-h\|_{\mathrm{cut}^s(Z)}= o_{\eta\to 0}(1)$.
\end{cor}
The second consequence is the analogue of Corollary \ref{c4.2} for hypergraphs.
\begin{cor} \label{c4.3}
Let $V$ be a nonempty finite set, and let $s\meg 2$ be an integer. Also let $0<\eta\mik 1$ and 
$\nu\colon V^s\to \rr^+$ such that $\|\nu-1\|_{\mathrm{cut}(V^s)}\mik \eta$. Then for every 
$G\colon V^s\to\rr$ with $0\mik G\mik \nu$ there exists $W\colon V^s\to [0,1]$ such that
$\|G-W\|_{\mathrm{cut}(V^s)}= o_{\eta\to 0}(1)$. Consequently, for every $F\colon V^s\to\rr$ 
with $|F|\mik \nu$ there exists $H\colon V^s\to [-1,1]$ such that $\|F-H\|_{\mathrm{cut}(V^s)}= o_{\eta\to 0}(1)$.
\end{cor}
Corollary \ref{c4.3} is a straightforward consequence of Proposition \ref{p4.1}. On the other hand, 
Corollary \ref{c4.2} follows by applying Proposition \ref{p4.1} for the family $\mathcal{F}$ of all 
convex combinations\footnote{The need to convexify the set of ``dual" functions is very natural from 
a functional analytic perspective; see, \textit{e.g.}, \cite{Go3}.} of functions $D\colon Z\to  \rr$ of the form
\[ D(z)=\ave\Big[ \prod_{\omega\in [2]^s\setminus \{1^s\}}\!\!\! H_\omega\big( \pi_\omega(x)\big) \,\Big|\,
x=(x_{ij})\in Z^{s\times 2} \text{ with } \sum_{i=1}^s x_{i1}=z \Big] \]
where $H_\omega\colon Z^s\to [-1,1]$ for every $\omega\in  [2]^s\setminus \{1^s\}$. Indeed, it is not hard to see that
this family $\mathcal{F}$ is closed under multiplication (see, \textit{e.g.}, the proof of Lemma 3.3 in~\cite{Zh}).

\subsection{The main results} \label{subsec4.3}

We are ready to state our first result in this section. It is a variant of \cite[Proposition~8.1]{GT1} 
(see also \cite[Proposition 10.3]{GT2}).
\begin{cor} \label{c4.4}
Let $Z$ be a finite additive group, let $s\meg 2$ be an integer, and let\, $0<\eta\mik 1$. 
Also let $\nu\colon Z\to\rr^+$ such that
\begin{equation} \label{e4.3}
\|\nu-1\|_{U^{2s}(Z)}\mik \eta.
\end{equation}
Then for every $f\colon Z\to\rr$ with $|f|\mik \nu$ there exists $h\colon Z\to [-1,1]$ such that
\begin{equation} \label{e4.4}
\|f-h\|_{U^s(Z)}=o_{\eta\to 0;s}(1).
\end{equation}
Moreover, if $f$ is nonnegative, then $h$ is also nonnegative.
\end{cor}
As we have already mentioned in the introduction, Corollary \ref{c4.4} answers a question of Gowers.
We also note that if $Z$ is a finite additive group and $f\colon Z\to \rr^+$ is a function which is approximated
by a $[0,1]$-valued function on $Z$ in the sense of \eqref{e4.4}---that is, there exists $h\colon Z\to [0,1]$ 
such that $\|f-h\|_{U^s(Z)}=o(1)$---then $f$ is majorized by a function $\nu\colon Z\to \rr^+$ satisfying 
$\|\nu-1\|_{U^s(Z)}=o(1)$; indeed, simply take $\nu\coloneqq f+(1-h)$. Thus we see that the pseudorandomness 
hypothesis \eqref{e4.3} is nearly optimal.
\begin{proof}[Proof of Corollary \emph{\ref{c4.4}}]
First observe that, by \eqref{e4.3}, the monotonicity of the Gowers norms 
$\|\cdot\|_{U^s(Z)}\mik \|\cdot\|_{U^{2s}(Z)}$, the identity \eqref{e3.2} and \eqref{e3.5}, 
we have that $\|\nu-1\|_{\mathrm{cut}^s(Z)}\mik \eta$. Hence, by Corollary \ref{c4.2}, 
there exists $h\colon Z\to [-1,1]$ such that $\|f-h\|_{\mathrm{cut}^s(Z)}= o_{\eta\to 0}(1)$. 
Set $\nu'\coloneqq (\nu+1)/2$ and notice that $|f-h|/2\mik \nu'$ and $\|\nu'-1\|_{U^{2s}(Z)}\mik \eta$. 
By Proposition~\ref{p2.1}, the result follows.
\end{proof}
Our second result is a variant of \cite[Theorem 3.9]{Tao}.
\begin{cor} \label{c4.5}
Let $V$ be a nonempty finite set, let $s\meg 2$ be an integer, and let\, $0<\eta\mik 1$. 
Also let $\nu\colon V^s\to\rr^+$ such that
\begin{equation} \label{e4.5}
\|\nu-1\|_{\square_4(V^s)}\mik \eta.
\end{equation}
Then for every $F\colon V^s\to\rr$ with $|F|\mik \nu$ there exists $H\colon Z\to [-1,1]$ such that
\begin{equation} \label{e4.6}
\|F-H\|_{\square(V^s)}=o_{\eta\to 0;s}(1).
\end{equation}
Moreover, if $F$ is nonnegative, then $H$ is also nonnegative.
\end{cor}
\begin{proof}
It is identical to the proof of Corollary \ref{c4.4}. Indeed, by \eqref{e4.5} and Corollary \ref{c4.3}, 
there exists $H\colon V^s\to [-1,1]$ such that $\|F-H\|_{\mathrm{cut}(V^s)}= o_{\eta\to 0}(1)$.
By Proposition \ref{p3.1}, the result follows.
\end{proof}


\section{On the relative inverse theorem for the Gowers $U^s[N]$-norm} \label{sec5}

\numberwithin{equation}{section}

\subsection{Overview} \label{subsec5.1}

In order to put the main result of this section in a proper context, we begin with a brief discussion 
on the \textit{nilpotent Hardy--Littlewood method} invented by Green and Tao \cite{GT2}. It is a powerful 
method for obtaining precise asymptotic estimates (as $N\to+\infty$) for expressions of the form
\begin{equation} \label{e5.1}
\sum_{n\in K\cap \zz^{d}} \, \prod_{i=1}^t f_i\big(\psi_i(n)\big)
\end{equation}
where $f_1,\dots,f_t\colon \zz\to\rr$ are arithmetic functions supported on the set of positive integers,
$K\subseteq [-N,N]^d$ is a convex body and $\psi_1,\dots,\psi_t\colon \zz^d\to~\zz$ are affine linear 
forms no two of which are affinely dependent. The first step of the method relies on the 
\textit{generalized von Neumann theorem}---see \cite[Proposition 7.1]{GT2}---which reduces 
the estimation of the quantity in \eqref{e5.1} to a norm estimate
\begin{equation} \label{e5.2}
\|f_i-1\|_{U^s[N]}=o_s(1) \ \ \text{for every } s\meg 2 \text{ and every } i\in\{1,\dots,t\}
\end{equation}
where $\|\cdot\|_{U^s[N]}$ stands for the $s$-th Gowers uniformity norm on the interval $[N]$ 
which we will shortly recall. This reduction can be performed provided that $|f_1|,\dots,|f_t|$ 
are simultaneously majorized by a function $\nu$ satisfying the ``linear forms condition" 
(see \cite[Definition 6.2]{GT2}). The second (and more substantial) step of the method reduces
the estimate \eqref{e5.2} to a non-correlation estimate
\begin{equation} \label{e5.3}
\ave_{n\in [N]} \, (f_i(n)-1) F(g^n\cdot x) =o_{s,G/\Gamma,M}(1)
\end{equation}
where $G/\Gamma$ is an $(s-1)$-step nilmanifold equipped with a smooth Riemannian metric~$d_{G/\Gamma}$,
$F\colon G/\Gamma\to [-1,1]$ is a function with Lipschitz constant at most $M$, $g\in G$ and $x\in G/\Gamma$.
(We recall the notion of an $(s-1)$-step nilmanifold below.) For \textit{bounded} functions, 
the equivalence between \eqref{e5.2} and \eqref{e5.3} is a deep result which is known as the 
\textit{inverse theorem for the Gowers $U^s[N]$-norm} and is due to Green, Tao and Ziegler \cite{GTZ}. 
One of the main steps in \cite{GT2} was to transfer the inverse theorem to the unbounded setting.
This was achieved with the \textit{relative inverse theorem for the 
Gowers $U^s[N]$-norm}---see \cite[Proposition 10.1]{GT2}---which can be applied  provided that 
$|f_i|$ is majorized by a function $\nu$ satisfying the aforementioned linear forms condition
and an additional pseudorandomness condition known as the ``correlation condition" (see \cite[Definition 6.3]{GT2}).

Recently, a part of the proof of \cite[Proposition 10.1]{GT2} was revisited in \cite{TZ1}. 
One pleasant consequence of the approach in \cite{TZ1} is that the relative inverse theorem 
(and, consequently, the whole nilpotent Hardy--Littlewood method) can be applied assuming 
that the majorant $\nu$ satisfies only the linear forms condition\footnote{The possibility 
that one could dispense with the need for the correlation condition entirely, 
was also noted in \cite[Appendix A]{FGKT}.}.

Our aim in this section is to give yet another proof of the relative inverse theorem using 
a norm-type pseudorandomness condition. To this end, it is convenient at this point to properly 
introduce the concepts discussed so far.

\subsubsection{Uniformity norms on intervals} \label{subsubsec5.1.1}

Let $N\meg 1$ be an integer, and let $f\colon [N]\to\rr$ be a function. We select an integer $N'>2N$ and we identify
(in the obvious way) the discrete interval $[N]$ with a subset of the cyclic group $\zz_{N'}\coloneqq \zz/N'\zz$.
The \textit{Gowers uniformity norm  $\|f\|_{U^s[N]}$ of\, $f$ on the interval $[N]$} is defined by setting
\begin{equation} \label{e5.4}
\|f\|_{U^s[N]}\coloneqq \|f\mathbf{1}_{[N]}\|_{U^s(\zz_{N'})}/ \|\mathbf{1}_{[N]}\|_{U^s(\zz_{N'})}
\end{equation}
where $\mathbf{1}_{[N]}\colon \zz_{N'}\to \{0,1\}$ stands for the indicator function of $[N]$. 
We note that the quantity $\|f\|_{U^s[N]}$ is, in fact, intrinsic and is independent of the choice 
of $N'$---see \cite[Appendix B]{GT2} for more details.

\subsubsection{Nilmanifolds} \label{subsubsec5.1.2}

Let $s\meg 2$ be an integer and recall that an \emph{$(s-1)$-step nilmanifold} is a homogeneous space 
$X\coloneqq G/\Gamma$ where $G$ is an $(s-1)$-step nilpotent, connected, simply connected Lie group, and 
$\Gamma$ is a discrete cocompact subgroup of~$G$. The group $G$ acts on $G/\Gamma$ by left multiplication 
and this action will be denoted by $(g,x)\mapsto g\cdot x$. As in \cite{GT2}, we will assume that each 
nilmanifold $G/\Gamma$ is equipped with a smooth Riemannian metric $d_{G/\Gamma}$; in particular,
if $F\colon G/\Gamma\to\rr$ is a function, then its Lipschitz constant is computed using the metric $d_{G/\Gamma}$.

\subsection{The main result} \label{subsec5.2}

We are ready to state the main result in this section. As we have indicated, 
it is a refinement\footnote{We notice that \cite[Proposition 10.1]{GT2} yields the 
existence of a finite family of nilmanifolds, but by taking their product, one 
can also formulate this result with a single nilmanifold; see, \textit{e.g.}, the remarks right after 
\cite[Conjecture 1.2]{GTZ}.} of \cite[Proposition 10.1]{GT2}.
\begin{thm} \label{t5.1}
For every integer $s\meg 2$, every $C\meg 20$ and every $0<\delta\mik 1$ there exist $\eta>0$, 
a constant $M>0$, a $(s-1)$-step nilmanifold $G/\Gamma$ equipped with a smooth Riemannian 
metric $d_{G/\Gamma}$, and a constant $c>0$ with the following property. Let $N$ be a 
positive integer, and let $N'\in [CN,2CN]$ be a prime. Also let $\nu\colon \zz_{N'}\to \rr^+$ satisfying
\begin{equation} \label{e5.5}
\|\nu-1\|_{U^{2s}(\zz_{N'})}\mik \eta.
\end{equation}
Finally, let $f\colon [N]\to\rr$ with $|f(n)|\mik \nu(n)$ for every $n\in [N]$. 
If\, $\|f\|_{U^s[N]}\meg \delta$, then there exist a function $F\colon G/\Gamma\to [-1,1]$ 
with Lipschitz constant at most $M$, $g\in G$, and $x\in G/\Gamma$ such that
\begin{equation} \label{e5.6}
|\ave_{n\in [N]}\, f(n) F(g^n\cdot x)| \meg c.
\end{equation}
\end{thm}
We notice that the estimate in \eqref{e5.5} follows if we assume that the function $\nu$ satisfies 
the $(4^s,4s,1)$-linear forms condition in the sense of \cite[Definition 6.2]{GT2}, but \eqref{e5.5} 
is certainly easier to grasp. It is likely that one can follow a similar approach in other instances 
of the transfer method, and replace the linear forms condition with a norm estimate of the form \eqref{e5.5} 
for a suitable uniformity norm\footnote{In this direction we recall (see also \cite{CFZ}) that it is not known
whether for every integer $k\meg 3$ there exists an integer $s\meg k-1$ such that the relative Szemer\'{e}di 
theorem for $k$-term arithmetic progressions holds true under the condition $\|\nu-1\|_{U^s(\zz_N)}=o(1)$.}.
\begin{rem} \label{r4}
Using Corollary \ref{c3.4} instead of Proposition \ref{p2.1}, it is easy to verify that Theorem \ref{t5.1}
also holds if the majorant $\nu$ satisfies $\|\nu-1\|_{U^s_4(\zz_{N'})}\mik \eta$, a condition which is 
slightly different from \eqref{e5.5}. However, the use of the $U^{2s}(\zz_{N'})\text{-norm}$ in Theorem \ref{t5.1} 
is conceptually more natural in the present arithmetic context.
\end{rem}

\subsection{Preliminary tools} \label{subsec5.3}

As in \cite{GT2}, the proof of Theorem \ref{t5.1} is based on three ingredients. The first one is 
the inverse theorem for the Gowers $U^s[N]$-norm~\cite{GTZ}. It gives a criterion for checking that 
a bounded arithmetic function has non-negligible uniformity norm.
\begin{thm} \label{t5.3}
For every integer $s\meg 2$ and every $0<\delta\mik 1$ there exist a constant $M>0$, 
a $(s-1)$-step nilmanifold $G/\Gamma$ equipped with a smooth Riemannian metric $d_{G/\Gamma}$,
and~a~constant $c>0$ with the following property. Let $N$ be a positive integer, and let $f\colon [N]\to [-1,1]$ 
such that $\|f\|_{U^s[N]}\meg \delta$. Then there exist a function $F\colon G/\Gamma\to [-1,1]$ 
with Lipschitz constant at most $M$, $g\in G$, and $x\in G/\Gamma$ such that
\begin{equation} \label{e5.7}
|\ave_{n\in [N]}\, f(n) F(g^n\cdot x)| \meg c.
\end{equation}
\end{thm}
It is more natural to formulate Theorem \ref{t5.3} for complex-valued functions which are bounded in magnitude 
by $1$; however, we will not need the complex version of Theorem \ref{t5.3} for the proof of Theorem \ref{t5.1}.

To state the second ingredient, we first recall some definitions. Let $s\meg 2$ be an integer. 
Also let $N$ be a positive integer, let $F\colon [N]\to\rr$ be a function, and define the
\textit{dual uniformity norm $\|F\|_{U^s[N]^*}$ of\, $F$} by the rule
\begin{equation} \label{e5.8}
\|F\|_{U^s[N]^*}\coloneqq \sup\big\{ |\ave_{n\in [N]}\, f(n)F(n)|: \|f\|_{U^s[N]}\mik 1\big\}.
\end{equation}
We will need the following result which follows from \cite[Proposition 11.2]{GT2}.
\begin{prop} \label{p5.4}
Let $s\meg 2$ be an integer, let $(G/\Gamma,d_{G/\Gamma})$ be an $(s-1)$-step nilmanifold, and let $M>0$.
Also let $F\colon G/\Gamma\to [-1,1]$ be a function with Lipschitz constant at most~$M$, $g\in G$, and $x\in G/\Gamma$.
Finally, let $N$ be a positive integer, and let $0<\ee\mik 1$. Then there exists a decomposition
\begin{equation} \label{e5.9}
F(g^n\cdot x)= F_1(n) + F_2(n) \ \ \text{for every } n\in [N]
\end{equation}
where the functions $F_1,F_2\colon [N]\to \rr$ obey the estimates
\begin{equation} \label{e5.10}
\|F_1\|_{\ell_\infty}=O(\ee) \ \text{ and } \ \|F_2\|_{U^s[N]^*}=O_{s,M,\ee,G/\Gamma}(1).
\end{equation}
\end{prop}
We point out that, by \cite[Proposition 11.2]{GT2}, one can additionally ensure that the function 
$F_2$ in the above decomposition is an ``averaged nilsequence" in the sense of \cite[Definition~11.1]{GT2}. 
We also note that the proof of \cite[Proposition 11.2]{GT2} is non-effective and yields no estimate 
for the dual uniformity norm of $F_2$. However, explicit estimates can be obtained by combining 
\cite[Lemmas A.2 and A.3]{M}---see \cite[Appendix~A]{M} for more details on this approach.

The last ingredient needed for the proof of Theorem \ref{t5.1} is the following version of 
Corollary \ref{c4.4} which concerns functions defined on intervals of $\zz$.
\begin{cor} \label{c5.5}
For every integer $s\meg 2$, every $C\meg 20$ and every $0<\ee\mik 1$ there exist a positive integer
$N_0$ and $\eta>0$ with the following property. Let $N\meg N_0$ be an integer, and let $N'\in [CN,2CN]$ be a prime. 
Also let $\nu\colon \zz_{N'}\to\rr^+$ satisfying
\begin{equation} \label{e5.11}
\|\nu-1\|_{U^{2s}(\zz_{N'})}\mik \eta.
\end{equation}
Finally, let $f\colon [N]\to\rr$ with $|f(n)|\mik \nu(n)$ for every $n\in [N]$. Then there exists 
a function $h\colon [N]\to [-1,1]$ such that
\begin{equation} \label{e5.12}
\|f-h\|_{U^s[N]} \mik \ee.
\end{equation}
Moreover, if $f$ is nonnegative, then $h$ is also nonnegative.
\end{cor}
\begin{proof}
It is a consequence of Corollary \ref{c4.4} and a standard truncation argument. 
Specifically, fix the parameters $s,C$ and $\ee$, and set
\begin{equation} \label{e5.13}
\alpha=\Big(\frac{\ee}{32 C}\Big)^{2^s} \ \text{ and } \ N_0=\lceil 2/\alpha\rceil.
\end{equation}
Moreover, by Corollary \ref{c4.4}, we select $0<\eta\mik 1$ such that for every finite additive group $Z$, 
every $\nu'\colon Z\to \rr^+$ satisfying $\|\nu'-1\|_{U^{2s}(Z)}\mik \eta$ and every $g\colon Z\to \rr$ 
with $|g|\mik \nu'$ there exists $w\colon Z\to [-1,1]$ such that $\|g-w\|_{U^s(Z)}\mik \ee\alpha/(32C)$. 
We will show that $N_0$ and $\eta$ are as desired.

So, let $N, N', \nu$ and $f$ be as in the statement of the corollary, and let $\tilde{f}\colon \zz_{N'}\to\rr$ 
be the extension of $f$ obtained by setting $\tilde{f}(n)=0$ if $n\notin [N]$. By the choice of $\eta$, 
there exists $H\colon \zz_{N'}\to [-1,1]$ satisfying
\begin{equation} \label{e5.14}
\|\tilde{f}-H\|_{U^s(\zz_{N'})}\mik \frac{\ee\alpha}{32C}.
\end{equation}
We claim that $\|f-h\|_{U^s[N]}\mik \ee$ where $h\colon [N]\to [-1,1]$ is the restriction of $H$ on $[N]$. 
Indeed, set $l=\lfloor \alpha N\rfloor$ and let $2L$ be the least even integer greater than or equal to $N$; 
notice that $N\meg L\meg l\meg 2$ and $\alpha/2\mik l/N\mik \alpha$. Next, write $N'=2k+1$ and identify $\zz_{N'}$ 
with the interval $\{-k,\dots,k\}$. Let $\varphi\colon \zz_{N'}\to [0,1]$ be the cut-off function which is nonzero 
on the set $\{-l+2,\dots,2L+l-1\}$, increases linearly from $0$ to~$1$ between $-l+1$ and $1$, is equal to $1$ on $[2L]$, 
and decreases linearly from $1$ to $0$ between $2L$ and $2L+l$. Observe that $\tilde{f}\varphi=\tilde{f}$ and so, 
setting $\tilde{h}\coloneqq H\mathbf{1}_{[N]}$, we have
\begin{equation} \label{e5.15}
\tilde{f}-\tilde{h}= (\tilde{f}-H)\varphi + H(\varphi-\mathbf{1}_{[N]}).
\end{equation}
Also note that the Fourier transform $\widehat{\varphi}$ of $\varphi$ satisfies the estimate
$\|\widehat{\varphi}\|_{\ell_1(\zz_{N'})}\mik 4L/l$ (see, \textit{e.g.}, the proof of Lemma A.1 in \cite{FH} 
where this is explained in some detail). Hence, by the triangle inequality and \cite[(11.11)]{TV},
we have\footnote{Note that here we work with the complex version of the Gowers uniformity norm.}
\begin{equation} \label{e5.16}
\|(\tilde{f}-H)\varphi\|_{U^s(\zz_{N'})} \mik \|\widehat{\varphi}\|_{\ell_1(\zz_{N'})} \cdot 
\|\tilde{f}-H\|_{U^s(\zz_{N'})} \mik \frac{4N}{l}\, \|\tilde{f}-H\|_{U^s(Z_{N'})}.
\end{equation}
On the other hand, since $H (\varphi-\mathbf{1}_{[N]})$ is bounded in magnitude by $1$ and is supported 
on a subset of $\zz_{N'}$ of cardinality at most $2l+1$, we obtain that
\begin{equation} \label{e5.17}
\|H(\varphi-\mathbf{1}_{[N]})\|_{U^s(Z_{N'})}\mik \Big(\frac {2l+1}{N'}\Big)^{1/2^s}\mik  \Big(\frac {3l}{C N}\Big)^{1/2^s}.
\end{equation}
Finally, note that $\|\mathbf{1}_{[N]}\|_{U^s(Z_{N'})}\!\meg \ave[\mathbf{1}_{[N]}] =N/N'\meg 1/2C$.  Thus,
by \eqref{e5.15}--\eqref{e5.17}, the triangle inequality and the definition of the $U^s[N]$-norm, we see that
\[ \|f-h\|_{U^s[N]}\mik 2C\, \Big( \frac{4N}{l}\, \|\tilde{f}-H\|_{U^s(Z_{N'})}+ \Big(\frac {3l}{C N}\Big)^{1/2^s}\Big). \]
By the previous inequality and taking into account the choice of $\alpha,l$ and the estimate \eqref{e5.14}, we conclude
that $\|f-h\|_{U^s[N]}\mik \ee$.
\end{proof}
\begin{rem} \label{r5.6}
We note that Corollary \ref{c5.5} also holds if the function $f$ is majorized 
by a function $\nu\colon [N]\to\rr^+$ which satisfies $\|\nu-1\|_{U^{2s}[N]}=o(1)$.
Indeed, given any integer $N'>2N$, the hypothesis $\|\nu-1\|_{U^{2s}[N]}=o(1)$ allows 
us to extend the function $\nu$ to a function $\nu'\colon \zz_{N'}\to\rr^+$ which also
satisfies $\|\nu'-1\|_{U^{2s}(\zz_{N'})}=o(1)$. (For instance, define $\nu'$ by setting 
$\nu'(n)=\nu(n)$ if $n\in [N]$ and $\nu'(n)=1$ otherwise.) Using this observation,
the desired approximation follows from Corollary \ref{c5.5}.
\end{rem}

\subsection{Proof of Theorem \ref{t5.1}} \label{subsec5.4}

We follow the proof from \cite[Proposition 10.1]{GT2} quite closely\footnote{Actually, there is a minor oversight 
in the proof of \cite[Proposition 10.1]{GT2} which is fixed in the present paper. Specifically, the appeal to 
Proposition 8.2 at the top of \cite[page 1796]{GT2} is invalid without appeal to the material from \cite[Section 11]{GT2}.}. 
We first observe that, by compactness, for every positive integer $d$ there exists a constant $D\meg 1$ such that for every 
$N\in [d]$ and every $f\colon [N]\to\rr$ we have that $\|f\|_{U^s[N]}\mik D\|\widehat{f}\|_{\ell_{\infty}}$. (Here, we identify 
$[N]$ with $\zz_N$.) Therefore, if $N\in [d]$, then Theorem \ref{t5.1} follows using as nilmanifold the torus $\rr/\zz$. 
Thus, at the cost of worsening the constants, it is enough to prove Theorem \ref{t5.1} for every sufficiently large
positive integer~$N$.

So, fix the parameters $s,C$ and $\delta$, and let $M, (G/\Gamma,d_{G/\Gamma})$ and $c$ be as in Theorem~\ref{t5.3} 
when applied for $\delta/2$. Next, by Proposition \ref{p5.4}, we select $K\meg 1$ such that for every function 
$F\colon G/\Gamma\to [-1,1]$ with Lipschitz constant at most $M$, every $g\in G$, every $x\in G/\Gamma$ 
and every integer $N\meg 1$ we have the decomposition \eqref{e5.9} with $\|F_1\|_{\ell_{\infty}}\mik c/12$ 
and $\|F_2\|_{U^s[N]^*}\mik K$. Finally, let $N_0$ and $\eta$ be as in Corollary \ref{c5.5} when applied 
for $\ee\coloneqq\min\{\delta/2,c/(4K)\}$. We claim that Theorem \ref{t5.1} holds true for 
$\eta,M, (G/\Gamma,d_{G/\Gamma})$ and $c/2$ provided that $N\meg N_0$. 

Indeed, let $N$ be an arbitrary positive integer with $N\meg N_0$, and let $N',\nu$ and $f$ be as in the 
statement of the theorem. By Corollary \ref{c5.5}, there exists $h\colon [N]\to [-1,1]$ such that
$\|f-h\|_{U^s[N]}\mik \ee$; in particular, we have $\|h\|_{U^s[N]}\meg \delta /2$ and so, by Theorem \ref{t5.3}, 
there exist a function $F\colon G/\Gamma\to [-1,1]$ with Lipschitz 
constant at most $M$, $g\in G$, and $x\in G/\Gamma$ such that
\begin{equation} \label{e5.18}
|\ave_{n\in [N]}\, h(n)F(g^n\cdot x)|\meg c.
\end{equation}
Write $F(g^n\cdot x)=F_1(n)+F_2(n)$ with $\|F_1\|_{\ell_{\infty}}\mik c/12$ and $\|F_2\|_{U^s[N]^*}\mik K$, 
and notice that, by \eqref{e5.18} and the triangle inequality, it suffices to show that
\begin{equation} \label{e5.19}
|\ave_{n\in [N]}\, \big(f(n)-h(n)\big)F_1(n)|\mik \frac{c}{4} \ \text{ and } \
|\ave_{n\in [N]}\, \big(f(n)-h(n)\big)F_2(n)|\mik \frac{c}{4}.
\end{equation}
The first part of \eqref{e5.19} follows from the fact that $\ave\big[|f-h|\big]\mik \ave[\nu+1]\mik 3$ and  
the fact that $\|F_1\|_{\ell_{\infty}}\mik c/12$. On the other hand, by the choice of $\ee$ and $h$, we have
\[ |\ave_{n\in [N]}\, \big(f(n)-h(n)\big)F_2(n)|\mik \|f-h\|_{U^s[N]}\cdot \|F_2\|_{U^s[N]^*}\mik \ee K \mik \frac{c}{4} \]
and the proof is completed.


\appendix

\section{Basic properties of uniformity norms} \label{appendix}

\numberwithin{equation}{section}

\begin{prop} \label{pa.1}
Let $V$ be a nonempty finite set and let $s\meg 2$ be an integer.
\begin{enumerate}
\item[(a)] \emph{$($Gowers--Cauchy--Schwarz inequality$)$} Let $\ell\meg 2$ be an even integer, 
and for every $\omega\in [\ell]^s$ let $F_\omega\colon V^s\to\rr$. Then we have
\begin{equation} \label{ea.1}
\Big| \ave\Big[\prod_{\omega\in [\ell]^s} F_\omega\big(\pi_\omega(x)\big) \Big|\, x\in V^{s\times \ell}\Big]\Big| \mik
\prod_{\omega\in [\ell]^s} \|F_\omega\|_{\square_\ell(V^s)}.
\end{equation}
In particular, if $Z$ is a finite additive group, then we have
\begin{equation} \label{ea.2}
\Big| \ave\Big[\prod_{\omega\in \{0,1\}^s} f_\omega(x+\omega\cdot\mathbf{h})\, \Big|\, x\in Z, 
\mathbf{h}\in Z^s\Big]\Big| \mik \prod_{\omega\in \{0,1\}^s} \|f_\omega\|_{U^s(Z)}
\end{equation}
for every family $\langle f_\omega:\omega\in \{0,1\}^s\rangle$ of real-valued functions on $Z$.
\item[(b)] For every even integer $\ell\meg 2$ the quantity $\|\cdot\|_{\square_\ell(V^s)}$ is a norm on $\rr^{V^s}$. 
Moreover, if\, $\ell_1\mik \ell_2$ are even positive integers, then for every $F\colon V^s\to\rr$ we have 
$\|F\|_{\square_{\ell_1}(V^s)}\mik \|F\|_{\square_{\ell_2}(V^s)}$.
\item[(c)] Let $\ell\meg 2$ be an even integer, let $0<\eta\mik 1$, and let $\nu \colon V^s \to \rr^+$ 
satisfying $\|\nu-1\|_{\square_{\ell+2}(V^s)}\mik \eta$. Then for every $F\colon V^s \to \rr$ with $|F|\mik \nu$ we have
\begin{equation} \label{ea.3}
\|F\|_{\square_\ell(V^s)}\mik \|F\|^{1/\ell^s}_{\square(V^s)}+ o_{\eta\to 0;s,\ell}(1).
\end{equation}
In particular, for every $F\colon V^s \to [-1,1]$ we have $\|F\|_{\square_\ell(V^s)}\mik \|F\|^{1/\ell^s}_{\square(V^s)}$.
\end{enumerate}
\end{prop}
\begin{proof}
Part (a) for $\ell=2$ is well-known (see \cite[Lemma B.2]{GT2} or \cite[Section~11.1]{TV}). 
The general case can be proved with similar arguments---see \cite[Proposition 2.1]{DKK2} for details. 
Part (b) is an easy consequence of the Gowers--Cauchy--Schwarz inequality.
Part (c) is a special (but more informative) case of \cite[Proposition 7.1]{GT2}. 
For the convenience of the reader we will sketch a proof.

We begin by introducing some pieces of notation. For every $\omega=(\omega_i) \in [\ell]^s$ 
we set $S(\omega)=\{i\in [s]: \omega_i=\ell\}$, and for every (possibly empty) $d\subseteq [s]$ 
let $\Omega'_{\omega,d}$ denote the set of all $\omega'=(\omega'_i)\in [\ell+1]^s$ such that
$\omega'_i\in \{\ell, \ell+1\}$ if $i\in S(\omega)\cap d$, and $\omega'_i=\omega_i$ otherwise. 
Next, for every $d\subseteq [s]$ let $I_d=([s]\times [\ell])\cup (d\times \{\ell+1\})$ and 
define\footnote{In this definition, as in the proof of Proposition \ref{p2.1}, we follow the 
convention that the product of an empty family of functions is equal to the constant function $1$.}
$F_d, G_d\colon V^{I_d}\to \rr$ by the rule
\[ F_d(x')=\!\!\!\! \prod_{\omega'\in \Omega'_{c,d}}\!\!\!\! F\big(\pi_{\omega'}(x')\big) \text{ \ and \ }
G_d(x')=\!\! \prod_{\omega\in A_d} \prod_{\omega'\in \Omega'_{\omega,d}}\!\!\!\!\! F\big(\pi_{\omega'}(x')\big) \!\!
\prod_{\omega\in B_d} \prod_{\omega' \in \Omega'_{\omega,d}}\!\!\!\! \nu\big(\pi_{\omega'}(x')\big) \]
where $c=(\ell,\dots,\ell)\in [\ell]^s$ denotes the sequence of length $s$ taking the constant value $\ell$,
$A_d=\{\omega\in [\ell]^s\setminus\{c\}:  d\subseteq S(\omega)\}$, 
$B_d=\{\omega\in [\ell]^s\setminus\{c\} :  d\nsubseteq S(\omega)\}$ and 
$\pi_{\omega'}(x')=(x'_{i\, \omega'_i})_{i=1}^s$ for every $x'\in V^{I_d}$ 
and every $\omega'=(\omega'_i)\in [\ell+1]^s$ such that 
$\{i\in [s]:\omega'_i=\ell+1\}\subseteq d$. Finally, we set $Q_d=\ave[F_d\, G_d]$.

Now observe that $Q_\emptyset=\ave[\,\prod_{\omega\in [\ell]^s } F\big(\pi_{\omega}(x)\big)\,|\, x\in V^{s\times \ell}]=
\|F\|_{\square_\ell(V^s)}^{\ell^s}$. Moreover,
\begin{eqnarray*}
Q_{[s]} & = &\ave\Big[ \prod_{\omega'\in \{\ell,\ell+1\}^s}\!\!\! F\big(\pi_{\omega'}(x')\big)
\prod_{\omega'\in [\ell+1]^s\setminus \{\ell,\ell+1\}^s}\!\!\! \nu\big(\pi_{\omega'}(x')\big) \,\Big|\, x'\in V^{s\times (\ell+1)}\Big] \\
&= & \|F\|_{\square(V^s)}^{2^s}+o_{\eta\to 0;s,\ell}(1).
\end{eqnarray*}
Indeed, write $Q_{[s]}=Q^{(1)}_{[s]}+Q^{(2)}_{[s]}$ where $Q^{(1)}_{[s]}=\ave\big[\, \prod_{\omega'\in \{\ell,\ell+1\}^s} F\big(\pi_{\omega'}(x')\big)\big]$ and
\[ Q^{(2)}_{[s]}=\ave\Big[ \prod_{\omega'\in \{\ell,\ell+1\}^s}\!\!\! F\big(\pi_{\omega'}(x')\big)
\cdot \Big( \prod_{\omega'\in [\ell+1]^s\setminus \{\ell,\ell+1\}^s}\!\!\!\!\!\!\!\!\!\!\!\nu\big(\pi_{\omega'}(x')\big) -1\Big)\Big]. \]
(Here, the first  expectation is taken over all  $x'\in V^{s\times \{\ell, \ell+1\}}$ and the second expectation is taken over all
$x'\in V^{s\times (\ell+1)}$.) Notice that $Q^{(1)}_{[s]}=\|F\|_{\square(V^s)}^{2^s}$. On the other hand, by a telescopic argument, the
Gowers--Cauchy--Schwarz inequality for the $\square_{\ell+2}(V^s)$-norm and the fact that $|F|\mik \nu$
and $\|\nu-1\|_{\square_{\ell+2}(V^s)}\mik \eta$, we obtain
\[ |Q_{[s]}^{(2)}|\mik \sum_{k=2^{s}+1}^{(\ell+1)^s} \|F\|_{\square_{\ell+2}(V^s)}^{2^s}\cdot
\|\nu-1\|_{\square_{\ell+2}(V^s)}\cdot \|\nu\|_{\square_{\ell+2}(V^s)}^{(\ell+1)^s-k}= o_{\eta\to 0;s,\ell}(1). \]
Finally, by repeated applications of the Cauchy--Schwarz inequality, we see that
\[ Q_d^2\mik \big(1+o_{\eta\to 0;s,\ell}(1)\big)\cdot Q_{d\cup \{i\}} \]
for every (possibly empty) $d\varsubsetneq  [s]$ and every $i\in [s]\setminus d$; in particular, 
we have that $Q_\emptyset^{2^s}\mik \big(1+o_{\eta\to 0; s,\ell}(1)\big)\cdot Q_{[s]}$. Since 
$Q_\emptyset=\|F\|_{\square_\ell(V^s)}^{\ell^s}$, $Q_{[s]}= \|F\|_{\square(V^s)}^{2^s}+o_{\eta\to 0; s,\ell}(1)$ and
\[ \|F\|_{\square(V^s)}\mik \|\nu\|_{\square(V^s)}\mik \|\nu\|_{\square_{\ell+2}(V^s)} \mik 1+\eta \]
the result follows.
%
%
\end{proof}

\subsection*{Acknowledgment} 

The research was supported by the Hellenic Foundation for Research and Innovation 
(H.F.R.I.) under the “2nd Call for H.F.R.I. Research Projects to support
Faculty Members \& Researchers” (Project Number: HFRI-FM20-02717).


\end{document}